\documentclass{birkjour_t2}
 \newtheorem{thm}{Theorem}[section]
 
 \newtheorem{lem}[thm]{Lemma}
 \newtheorem{prop}[thm]{Proposition}
 \theoremstyle{definition}
 \newtheorem{defn}[thm]{Definition}
 \theoremstyle{remark}
 
 \newtheorem*{ex}{Example}
 \numberwithin{equation}{section}

\catcode`\@=11
\def\omathop#1#2#3{\let\temp=#1\def\letter{#2}
  \ifcat#3_ \let\next\@@olim\else\let\next\@olim\fi\next#3}
\def\@olim{\letter\text{-}\!\temp}
\def\@@olim_#1{\mathchoice{
   \setbox0=\hbox{$\displaystyle\letter\text{-}\!\temp\!\text{-}\letter$}
   \setbox2=\hbox{$\displaystyle\temp$}
   \setbox4=\hbox{$\scriptstyle#1$}
   \dimen@=\wd4 \advance\dimen@ by -\wd2 \divide\dimen@ by2
   \def\next{\letter\text{-}\!\temp_{\hbox to 0pt{\hss$\scriptstyle#1$\hss}}
     \hskip\dimen@}
   \ifdim\wd2>\wd4 \def\next{\@olim_{#1}}\fi
   \ifdim\wd4>\wd0 \def\next{\mathop{\llap{$\letter$-}\!\temp}\limits_{#1}}\fi
   \next}
   {\@olim_{#1}}{\@olim_{#1}}{\@olim_{#1}}}

\def\bolim{\omathop{\lim}{bo}}

\catcode`\@=13

\newcommand{\reduce}{\mskip-2mu}
\newcommand{\ls}{\reduce\left\bracevert\reduce\vphantom{X}}
\newcommand{\rs}{\reduce\vphantom{X}\reduce\right\bracevert\reduce}

\begin{document}

%
%
%
%
%
%
%
%
%

 \title[Narrow operators on lattice-normed]
 {Narrow operators on lattice-normed\\ spaces and vector measures}

 \author[D. T. Dzadzaeva]{Dina Dzadzaeva}
\address{Department of Mathematics\\
and Informatics\\
North-Osetian State University\\
Vladikavkaz, Russia}


\thanks{This work was completed with the support of Russian Foundation of Basic Research, grant number  15-51-53119.}

\author[M. A. Pliev]{Marat Pliev}

\address{South Mathematical Institute\\
of the Russian Academy of Sciences\\
Vladikavkaz, Russia}




\date{June 30, 2014}

\dedicatory{Dedicated to the memory of Adriaan Cornelis Zaanen, the one of the great  founder of the theory of vector lattices}

\begin{abstract}
We consider linear narrow operators on lattice-normed spaces. We
prove that, under mild assumptions, every finite rank linear
operator is strictly narrow (before it was known  that such
operators are narrow). Then we show that every dominated, order
continuous linear operator from a lattice-normed space over atomless
vector lattice to an  atomic lattice-normed  space  is order narrow.
\end{abstract}

\maketitle

\section{Introduction}

Today the theory of narrow operators is a growing and active  field
of Functional Analysis (see the recent monograph \cite{PR}). Plichko
and Popov were the first \cite{PP90} who systematically studied this
class of operators. Later many authors  have  studied linear and
nonlinear narrow operators  in functional spaces and vector lattices
\cite{Bi-2, F, MMP, MPRS, PP}.  In the article \cite{P-1} the second
named author  have considered a general lattice-normed space
approach to narrow operators. Recently it became clear that a
technique of vector measures is relevant to narrow operators
\cite{MPPS}. The aim of this article is  to use this new technique
for investigation of linear narrow operators on lattice-normed
spaces.

\section{Preliminaries}

The  goal of this section is to introduce some basic definitions and
facts. General information on vector lattices, Banach spaces and
lattice-normed spaces  can be found in the books
\cite{AA,AB,Kus,LZ}.

Consider a vector space $V$ and a real  archimedean vector lattice
$E$. A map $\ls \cdot\rs:V\rightarrow E$ is a \textit{vector norm}
if it satisfies the following axioms:
\begin{enumerate}
  \item[1)] $\ls v \rs\geq 0;$\,\, $\ls v\rs=0\Leftrightarrow v=0$;\,\,$(\forall v\in V)$.
  \item[2)] $\ls v_1+v_2 \rs\leq \ls v_1\rs+\ls v_2 \rs;\,\, ( v_1,v_2\in V)$.
  \item[3)] $\ls\lambda  v\rs=|\lambda|\ls v\rs;\,\, (\lambda\in\Bbb{R},\,v\in V)$.
\end{enumerate}
A vector norm is called \textit{decomposable} if
\begin{enumerate}
  \item[4)] for all $e_{1},e_{2}\in E_{+}$ and $x\in V$ with $\ls x\rs=e_{1}+e_{2}$  there exist $x_{1},x_{2}\in V$ such that $x=x_{1}+x_{2}$ and $\ls x_{k}\rs=e_{k}$, $(k:=1,2)$.
\end{enumerate}

A triple $(V,\ls\cdot\rs,E)$ (in brief $(V,E),(V,\ls\cdot\rs)$ or
$V$ with default parameters omitted) is a \textit{lattice-normed
space} if $\ls\cdot\rs$ is a $E$-valued vector norm in the vector
space $V$. If the norm $\ls\cdot\rs$ is decomposable then the space
$V$ itself is called decomposable. A subspace $V_{0}$ of $V$ is
called a $\text{(bo)}$-ideal of $V$ if for $v\in V$ and $u\in
V_{0}$, from $\ls v\rs\leq\ls u\rs$ it follows that $v\in V_{0}$. We
say that a net $(v_{\alpha})_{\alpha\in\Delta}$ {\it
$(bo)$-converges} to an element $v\in V$ and write $v=\bolim
v_{\alpha}$ if there exists a decreasing net
$(e_{\gamma})_{\gamma\in\Gamma}$ in $E^{+}$ such that
$\inf_{\gamma\in\Gamma}(e_{\gamma})=0$ and for every
$\gamma\in\Gamma$ there is an index $\alpha(\gamma)\in\Delta$ such
that $\ls v-v_{\alpha(\gamma)}\rs\leq e_{\gamma}$ for all
$\alpha\geq\alpha(\gamma)$. A net $(v_{\alpha})_{\alpha\in\Delta}$
is called \textit{$(bo)$-fundamental} if the net
$(v_{\alpha}-v_{\beta})_{(\alpha,\beta)\in\Delta\times\Delta}$
$(bo)$-converges to zero. A lattice-normed space is called {\it
$(bo)$-complete} if every $(bo)$-fundamental net $(bo)$-converges to
an element of this space.  Every decomposable $(bo)$-complete
lattice-normed space is called a {\it Banach-Kantorovich space} (a
BKS for short).

Let $V$ be a lattice-normed space and $y,x\in V$. If $\ls
x\rs\wedge\ls y\rs=0$ then we call the elements $x,y$ {\it disjoint}
and write $x\bot y$. As in the case of a vector lattice, a set of
the form $M^{\bot}=\{v\in V:(\forall u\in M)u\bot v\}$, with
$\emptyset\neq M\subset V$, is called a \textit{band}. The equality
$x=\coprod_{i=1}^{n}x_{i}$ means that $x=\sum\limits_{i=1}^{n}x_{i}$
and $x_{i}\bot x_{j}$ if $i\neq j$. An element $z\in V$ is called a
{\it component} or a \textit{fragment} of $x\in V$ if  $z\bot(x-z)$.
Two fragments $x_{1},x_{2}$ of $x$  are called \textit{mutually
complemented} or $MC$, in short, if $x=x_1+x_{2}$. The notations
$z\sqsubseteq x$ means that $z$ is a fragment of $x$. The set of all
fragments of the element $v\in V$ is  denoted by $\mathfrak{F}_{v}$.
Following (\cite{AB},\,p.111) an element $e>0$ of a vector lattice
$E$ is called an {\it atom}, whenever $0\leq f_{1}\leq e$, $0\leq
f_{2}\leq e$ and $f_{1}\bot f_{2}$ imply that either $f_{1}=0$ or
$f_{2}=0$. A vector lattice $E$ is  atomless if there is no atom
$e\in E$.

Remark that vector lattices and normed spaces are simple examples of
lattice normed spaces. Indeed, if $V=E$ then the modules of an
element can be taken as its lattice norm: $\ls
v\rs:=|v|=v\vee(-v);\,v\in E$. Decomposability of this norm easily
follows from the Riesz Decomposition Property holding in every
vector lattice. If $E=\Bbb{R}$ then $V$ is a normed space. Many
nontrivial examples of lattice-normed spaces a reader can be found
in the book \cite{Kus}.

Let $E$ be a Banach lattice and let $(V,E)$ be a lattice-normed
space. Since $\ls x\rs\in E_{+}$ for every $x\in V$  we can define a
\textit{mixed norm} in $V$ by the formula
$$
\||x|\|:=\|\ls x\rs\|\,\,\,(\forall\, x\in V).
$$
The normed space $(V,\||\cdot|\|)$ is called a \textit{space with a
mixed norm}. In view of the inequality $|\ls x\rs-\ls y\rs|\leq\ls
x-y\rs$ and monotonicity of the norm in $E$, we have
$$
\|\ls x\rs-\ls y\rs\|\leq\||x-y|\|\,\,\,(\forall\, x,y\in V),
$$
so a vector norm is a norm continuous operator from
$(V,\||\cdot|\|)$ to $E$. A lattice-normed space $(V,E)$ is called a
\textit{Banach space with a mixed norm} if the normed space
$(V,\||\cdot|\|)$ is complete with respect to the norm convergence.

Consider lattice-normed spaces $(V,E)$ and $(W,F)$, a linear
operator $T:V\rightarrow W$ and a positive operator $S\in
L_{+}(E,\,F)$. If the condition
$$
\ls Tv\rs\leq S\ls v\rs;\,(\forall\, v\in V)
$$
is satisfied then we say that $S$ \textit{dominates} or
\textit{majorizes} $T$ or that $S$ is \textit{dominant} or
{majorant} for $T$. In this case $T$ is called a \textit{dominated}
or \textit{majorizable} operator.  The set of all dominants of the
operator $T$ is denoted by $\text{maj}(T)$. If there is the least
element in $\text{maj}(T)$ with respect to the order induced by
$L_{+}(E,F)$ then it is called the {\it least} or the {\it exact
dominant} of $T$ and it is denoted by $\ls T\rs$.

We follow \cite{P-1} in the next definition.

\begin{defn} \label{def:nar1}
Let $(V,E)$ be an lattice-normed space over an atomless vector
lattice $E$ and $X$ a vector space. A linear operator $T: V \to X$
is called:
\begin{itemize}
  \item \emph{strictly narrow} if for every $v \in V$ there exists a decomposition $v = v_{1} \sqcup v_{2}$ of $v$ such that $T(v_{1}) = T(v_{2})$;
  \item \emph{narrow} if $X$ is a normed space, and for every $v \in V$ and every $\varepsilon > 0$ there exists a decomposition $v = v_{1} \sqcup v_{2}$ of $v$ such that $\|T(v_1 - v_2)\| < \varepsilon$;
  \item \emph{order narrow} if $X$ is a Banach space with a mixed norm, and for every $v \in V$ there exists a net of decompositions $v = v_{\alpha}^{1} \sqcup v_{\alpha}^{2}$ such that $T(v_{\alpha}^{1}-v_{\alpha}^{2})\overset{(bo)}\longrightarrow 0$.
\end{itemize}
\end{defn}

A linear operator $T$ from a lattice-normed space $V$ to a Banach
space $X$ is called:
\begin{itemize}
  \item \textit{order-to-norm} $\sigma$-continuous if $T$ sends $(bo)$-convergent sequences in $V$ to norm convergent sequences in $X$;
  \item \textit{order-to-norm} continuous provided $T$ sends $(bo)$-convergent nets in $V$ to norm convergent nets in $X$.
\end{itemize}

Necessary information on Boolean algebras  can be found, for
instance, in \cite{Jec}, \cite{Kus}, \cite{LZ}. The most common
example of a Boolean algebra is an algebra $\mathcal A$ of subsets
of a set $\Omega$, that is, a subset of the set $\mathcal P(\Omega)$
of all subsets of $\Omega$, closed under the union, intersection and
complementation and containing $\emptyset$ and $\Omega$. The Boolean
operations on $\mathcal A$ are $A \boldsymbol{\vee} B = A \cup B$;
$A \boldsymbol{\wedge} B = A \cap B$ and $\neg A = \Omega \setminus
A$, and the constants are $\mathbf{0} = \emptyset$, $\mathbf{1} =
\Omega$.

A map $h: \mathcal A \to \mathcal B$ between two Boolean algebras is
called a \emph{Boolean homomorphism} if the following conditions
hold for all $x,y \in \mathcal A$
\begin{enumerate}
  \item $h(\mathbf{0}) = \mathbf{0}$;
  \item $h(\mathbf{1}) = \mathbf{1}$;
  \item $h(x \boldsymbol{\vee} y) = h(x) \boldsymbol{\vee} h(y)$;
  \item $h(x \boldsymbol{\wedge} y) = h(x) \boldsymbol{\wedge} h(y)$;
  \item $h(\neg x) = \neg h(x)$.
\end{enumerate}

A bijective Boolean homomorphism which is called a \emph{Boolean
isomorphism}. Two Boolean algebras $\mathcal A$ and $\mathcal B$ are
called \emph{Boolean isomorphic} if there is a Boolean isomorphism
$h: \mathcal A \to \mathcal B$. The following remarkable result is
known as the Stone representation theorem.

\begin{thm} \label{th:Stone} {\rm(\cite[Theorem~7.11]{Jec})}
Every Boolean algebra is Boolean isomorphic to an algebra of subsets
of some set.
\end{thm}

Every  Boolean algebra $\mathcal A$ is a partially ordered set with
respect to the partial order $``x \leq y$ if and only if $x
\boldsymbol{\wedge} y = x"$, with respect to which $\mathbf{0}$ is
the least element, $\mathbf{1}$ is the greatest element, $x
\boldsymbol{\wedge} y$ is the infimum and $x \boldsymbol{\vee} y$
the supremum of the two-point set $\{x,y\}$ in $\mathcal A$. A
Boolean algebra $\mathcal A$ is called Dedekind complete (resp.,
$\sigma$-Dedekind complete) if so is $\mathcal A$ as a partially
ordered set, that is, if every (resp., countable) order bounded
nonempty subset of $\mathcal A$ has the least upper and the greatest
lower bounds in $\mathcal A$. Obviously, a Boolean algebra is
$\sigma$-Dedekind complete if and only if it is a $\sigma$-algebra.
A Boolean algebras $A$ can be viewed as an algebra over field
$\Bbb{Z}_{2}:=\{\boldsymbol{0},\boldsymbol{1}\}$ with respect of an
algebraic operations:
$$
x+y:=x\triangle y,\,xy:=x\wedge y \,\,(x,y\in A),
$$
where $x\triangle y:=(x\wedge y^{\star})\vee(x^{\star}\wedge y)$ is
a {\it symmetric difference} of $x$ and $y$.

There is a natural connection between Boolean algebras and
lattice-normed spaces. Let $(V,E)$ be a lattice-normed space. Given
$L\subset E$ and $M\subset V$, we let by definition $h(L)=\{v\in
V:\,\ls v\rs\in L\}$ and $\ls M\rs=\{\ls v\rs:\, v\in M\}$. It is
clear that $\ls h(L)\rs\subset L\bigcap\ls V\rs$.

\begin{prop}\label{lat-norm} {\rm(\cite[Proposition~2.1.2]{Kus})}
Suppose that every band of vector lattice $E_{0}=\ls
V\rs^{\bot\bot}$ contains the norm of some nonzero element. Then the
set of all bands of the lattice-normed space $V$ is a complete
Boolean algebra and the map $L\mapsto h(L)$ is an isomorphism of the
Boolean algebras of bands of $E_{0}$ and $V$.
\end{prop}
All lattice-normed spaces  considered below are decomposable and
satisfy the proposition \ref{lat-norm}.

\section{Vector measures on Boolean algebras}

By a \emph{measure} on a Boolean algebra $\mathcal A$ we mean a
finitely additive function $\mu: \mathcal A \to X$ of $\mathcal A$
to a vector space $X$, that is, a map satisfying
$$
(\forall x,y \in \mathcal A) \Bigl( \bigl( x \boldsymbol{\wedge} y =
\mathbf{0} \bigr) \Rightarrow \bigl( \mu(x + y) = \mu(x) + \mu(y)
\bigr) \Bigr).
$$

If, moreover, $\mathcal A$ is a Boolean $\sigma$-algebra and $X$ is
a topological vector space then a $\sigma$-\emph{additive measure}
is a measure $\mu: \mathcal A \to X$ possessing the property that if
$(x_n)_{n=1}^\infty$ is a sequence in $\mathcal A$ with $x_n
\uparrow x \in \mathcal A$ then $\lim\limits_{n \to \infty} \mu(x_n)
= \mu(x)$.

\subsection{Definitions and simple properties.}
We follow \cite{MPPS} in  the definitions below.

\begin{defn}
Let $\mathcal A$ be a Boolean algebra and $X$ a normed space. A
measure $\mu: \mathcal A \to X$ is called \emph{almost dividing} if
for every $x \in \mathcal A$ and every $\varepsilon > 0$ there is a
decomposition $x = y \sqcup z$ with $\|\mu(y) - \mu(z)\| <
\varepsilon$.
\end{defn}

\begin{defn}
Let $\mathcal A$ be a Boolean algebra and $V$ a lattice-normed
space. A measure $\mu: \mathcal A \to V$ is called \emph{order
dividing} if for every $x \in \mathcal A$ there is a net of
decompositions $x = y_\alpha \sqcup z_\alpha$ with
$\bigl(\mu(y_\alpha) - \mu(z_\alpha) \bigr)
\stackrel{(bo)}{\longrightarrow} 0$.
\end{defn}

\begin{defn}
Let $\mathcal A$ be a Boolean algebra and $X$ a vector space. A
measure $\mu: \mathcal A \to X$ is called \emph{dividing} if for
every $x \in \mathcal A$ there is a decomposition $x = y \sqcup z$
with $\mu(y) = \mu(z)$.
\end{defn}

\begin{defn} \label{def:ass}
Let $V$ be a lattice-normed space, $X$ a vector space. To every
linear  operator $T: V \to X$ we associate a family of measures
$(\mu_v^{_T})_{v \in V}$ as follows. Given any $v \in V$, we define
a measure $\mu_v^{_T}: \mathfrak{F}_v \to X$ on the Boolean algebra
$\mathfrak{F}_v$ of fragments of $v$ by setting $\mu_v^{_T} u =
T(u)$, $\mu_v^{_T}$ is called the \emph{associated measure} of $T$
at $v$.
\end{defn}

The next proposition directly follows from the definitions.

\begin{prop} \label{pr:dirfol}
Let $(V,E)$ be a lattice-normed space, $E$ be an atomless vector
lattice, $X$ a vector space and $T: V \to X$ a linear operator. Then
the following assertions hold.
\begin{enumerate}
  \item[$(1)$] $T$ is strictly narrow if and only if for every $v \in V$ the measure $\mu_v^{_T}$ is dividing.
  \item[$(2)$] Let $X$ be a normed space. Then $T$ is narrow if and only if for every $v \in V$ the measure $\mu_v^{_T}$ is almost dividing.
  \item[$(3)$] Let $X$ be a lattice-normed space. Then $T$ is order narrow if and only if for every $v \in V$ the measure $\mu_v^{_T}$ is order dividing.
\end{enumerate}
\end{prop}

Obviously, a dividing measure is both almost dividing and order
dividing, for an appropriate range space. The following three
propositions are close to propositions~10.7 and~10.9, and
Example~10.8 from \cite{PR}.

\begin{prop} \label{pr:10.7}
Let $\mathcal A$ be a Boolean algebra and $W$ a Banach space with a
mixed norm. Then every almost dividing measure $\mu: \mathcal A \to
W$ is order dividing.
\end{prop}

\begin{proof}
Let $\mu: \mathcal A \to W$ be an almost dividing measure and $x \in
\mathcal A$. Choose a sequence of decompositions $x = y_n \sqcup
z_n$ with $|\|\mu(y_n) - \mu(z_n)|\| \leq 2^{-n}$. Then for $u_n =
\sum_{k=n}^\infty \ls\mu(y_n) - \mu(z_n)\rs$ one has $\ls\mu(y_n) -
\mu(z_n)\rs \leq u_n \downarrow 0$. Hence, $\bigl(\mu(y_n) -
\mu(z_n) \bigr) \stackrel{(bo)}{\longrightarrow} 0$.
\end{proof}

\begin{prop} \label{pr:10.8}
Let $\Sigma$ be the Boolean $\sigma$-algebra of Lebesgue measurable
subsets of $[0,1]$. Then there exists an order dividing measure
$\mu: \Sigma \to L_\infty$ which is not dividing.
\end{prop}

\begin{proof}
Use Proposition~\ref{pr:dirfol} and \cite[Example~10.8]{PR}.
\end{proof}

\begin{prop} \label{pr:10.9}
Let $\mathcal A$ be a Boolean algebra and $(W,F)$ a Banach space
with a mixed norm, where $F$ is   an order continuous Banach
lattice. Then a measure $\mu: \mathcal A \to W$ is order dividing if
and only if it is almost dividing.
\end{prop}

\begin{proof}
Let $\mu: \mathcal A \to W$ be order dividing. Given any $x \in
\mathcal A$, let $x = y_\alpha \sqcup z_\alpha$ be a net of
decompositions with $\bigl(\mu(y_\alpha) - \mu(z_\alpha) \bigr)
\stackrel{(bo)}{\longrightarrow} 0$. By the order continuity of $F$,
$\bigl\|\ls \mu(y_\alpha) - \mu(z_\alpha) \rs\bigr\| \to 0$, and
hence, $\mu$ is almost dividing by arbitrariness of $x \in \mathcal
A$. By Proposition~\ref{pr:10.7}, the proof is completed.
\end{proof}

A nonzero element $u$ of a Boolean algebra $\mathcal A$ is called an
\emph{atom} if for every $x \in \mathcal A$ the condition $0 < x
\leq u$ implies that $x = u$. Every dividing (of any type) measure
sends atoms to zero.

\begin{prop} \label{pr:atom}
Let $\mathcal A$ be a Boolean algebra and $V$ a vector space (a
normed space, or a lattice-normed space) and $\mu: \mathcal A \to V$
a dividing (an almost dividing or an order dividing, respectively)
measure. If $a \in \mathcal A$ is an atom then $\mu(a) = 0$.
\end{prop}

The proof is an easy exercise.

\subsection{The range convexity of vector measures}

We need the following remarkable result known as the
Lyapunov\footnote{= Lyapounoff, the old spelling} convexity theorem.

\begin{thm}\label{th:Lyap} {\rm\cite[Theorem~2, p.9]{LTII}}
Let $(\Omega, \Sigma)$ be a measurable space, $X$ a finite
dimensional normed space and $\mu: \Sigma \to X$ an atomless
$\sigma$-additive measure. Then the range $\mu(\Sigma) = \{\mu(A):
\, A \in \Sigma\}$ of $\mu$ is a compact convex subset of $X$.
\end{thm}

The following theorem was proven in \cite{MPPS}, but for sake of
completeness we include the proof here.

\begin{thm} \label{thm:first}
Let $\mathcal A$ be a Boolean $\sigma$-algebra and $X$ a finite
dimensional vector space. Then every atomless $\sigma$-additive
measure $\mu: \mathcal A \to X$ is dividing.
\end{thm}

For a Boolean $\sigma$-algebra $\mathcal A$ and $x \in \mathcal
A\setminus\{0\}$ by $\mathcal A_x$ we denote the  Boolean
$\sigma$-algebra $ \{y \in \mathcal A: \, y \leq x\}$ with the unit
$\mathbf{1}_{\mathcal A_x} = x$ and the operations induced by
$\mathcal A$.

\begin{proof}
Let $\mu: \mathcal A \to X$ be an atomless $\sigma$-additive measure
and $x \in \mathcal A$. If $x = 0$ then there is nothing to prove.
Let $x \neq 0$. Then the restriction $\mu_x = \mu|_{\mathcal A_x}:
\mathcal A_x \to X$ is an atomless $\sigma$-additive measure. By
Theorem~\ref{th:Stone}, $\mathcal A_x$ is Boolean isomorphic to some
measurable space $(\Omega, \Sigma)$ by means of some Boolean
isomorphism $J: \mathcal A_x \to \Sigma$. Since $\mathcal A_x$ is a
Boolean $\sigma$-algebra, $\Sigma$ is a $\sigma$-algebra. Then the
map $\nu: \Sigma \to X$ given by $\nu(A) = \mu \bigl( J^{-1}(A)
\bigr)$ for all $A \in \Sigma$, is an atomless $\sigma$-additive
measure. By Theorem~\ref{th:Lyap}, the range $\nu(\Sigma)$ of $\nu$
is a convex subset of $X$. In particular, since $0, \nu
\bigl(J(x)\bigr) \in \nu(\Sigma)$, we have that $\nu
\bigl(J(x)\bigr)/2 \in \nu(\Sigma)$. Let $B \in \Sigma$ be such that
$\nu(B) = \nu \bigl(J(x)\bigr)/2 = \mu(x)/2$. Then for $y =
J^{-1}(B)$ one has that $y \leq x$ and $\mu(y) = \nu(B) = \mu(x)/2$.
Thus, for $z = x \wedge \neg y$ one has $x = y \sqcup z$ and $\mu(z)
= \mu(x) - \mu(y) = \mu(x)/2 = \mu(y)$.
\end{proof}

\subsection{Strict narrowness of order continuous finite rank operators}

The following theorem is the main result of this subsection. This
assertion strengthens Theorem~4.12 from \cite{P-1}. Using a  method
based on the Lyapunov theorem, we prove the strict narrowness of an
operator.

\begin{thm} \label{pr:relatedd8}
Let $(V,E)$ be a lattice-normed space, $E$  an atomless vector
lattice  with the principal projection property, $X$ a finite
dimensional normed space (resp., a lattice-normed space ). Then
every $\sigma$-order-to-norm continuous (resp., $\sigma$-order
continuous) linear operator $T: V \to X$ is strictly narrow.
\end{thm}

To use the technique of dividing vector measures, we preliminarily
need the $\sigma$-additivity of a measure.

\begin{lem} \label{pr:later}
Let $(V,E)$ be a lattice-normed space, $E$  an atomless vector
lattice  with the principal projection property,  $X$ a normed space
(resp., a lattice-normed  space), $T: V \to X$ an order-to-norm
continuous (resp., a order continuous) linear operator. Then for
every $v\in V$ the associated measure $\mu_v^{_T}$ is atomless and
$\sigma$-additive.
\end{lem}

\begin{proof}[Proof of Lemma~\ref{pr:later}]
Fix any $v\in V$. The $\sigma$-additivity of $\mu_v^{_T}$ directly
follows from the order  continuity. We show that $\mu_v^{_T}$ is
atomless. Assume $v_0 \in \mathfrak{F}_v$ and $\mu_v^{_T}(v_0) \neq
0$, that is, $T(v_0) \neq 0$. Set $Z = \{u \in \mathfrak{F}_{x_0}:
\,\, T(u) = 0\}$. By the order continuity and Zorn's lemma, $Z$ has
a maximal element $z \in Z$. Since $T(z) = 0$, one has that $T(v_0 -
z) = T(z) + T(v_0 - z) = T(v_0) \neq 0$. Since $E$ is atomless, we
split $v_0 - z = w_{1} \sqcup w_{2}$ with $w_{1},w_{2} \in
\mathfrak{F}_{v_0} \setminus \{0\}$. By maximality of $z$, $T(w_{1})
\neq 0$ and $T(w_{2}) \neq 0$. Thus, $v_0 = (z + w_{1}) \sqcup
w_{2}$ is a decomposition with $\mu_v^{_T}(z + w_{1}) =
\mu_v^{_T}(w_{2}) \neq 0$ and $\mu_e^{_T} (w_{1}) \neq 0$.
\end{proof}

\begin{proof}[Proof of Theorem~\ref{pr:relatedd8}]
Let $v \in V$. By Lemma~\ref{pr:later}, the associated measure
$\mu_v^{_T}: \mathfrak{F}_v \to X$ is atomless and
$\sigma$-additive. By Theorem~\ref{thm:first}, $\mu_v^{_T}$ is
dividing. So, we split $v = v_{1} \sqcup v_{2}$ with
$\mu_v^{_T}(v_{1}) = \mu_v^{_T}(v_{2})$, that is, $T(v_{1}) =
T(v_{2})$.
\end{proof}

\section{Operators from arbitrary  to  atomic lattice-normed spaces  are order narrow}

\begin{defn}
An element $u$ of a lattice-normed space  $(V,E)$ is called an
\textit{atom}, whenever $0 \leq \ls v_{1}\rs \leq \ls u\rs$, $0
\leq\ls v_{2}\rs  \leq\ls u\rs$ and $v_{1}\bot v_{2}$ imply that
either $v_{1} = 0$ or $v_{2} = 0$.
\end{defn}
It is clear that $u\bot v$ for every  different atoms $u,v\in V$.
\begin{defn}
A lattice-normed space $V$ is said to be \textit{ atomic} if there
is a collection $(u_i)_{i \in I}$ of atoms in $V$, called a
\emph{generating collection of atoms}, such that $u_i \bot u_j$ for
$i \neq j$ and for every $v \in V$ if $\ls v\rs \wedge \ls u_i\rs =
0$ for each $i \in I$ then $v = 0$.
\end{defn}
It follows from the definition that the $(bo)$-ideal $V_{0}$
generated by generating collection of atoms coincides with $V$.
\begin{ex}
Let $X$ be a Banach space and consider the lattice-normed space
$(V,E)$, where $E=c_{0}$ or $E=l^{p}$; $0<p\leq\infty$ and
$$
V=\{(x_{n})_{n=1}^{\infty};\,x_{n}\in
X:\,(\|x_{n}\|)_{n=1}^{\infty}\in E\}.
$$
Elements $\{(0,\dots,x_{i},\dots,0,\dots):\,x_{i}\in
X;\,i\in\Bbb{N}\}$ are atoms and the space $V$ is an atomic
lattice-normed space.
\end{ex}

\begin{prop}\label{atom}
Let  $(V,E)$ be a  lattice-normed space such that  every
band of the vector lattice $E$ contains the vector norm of some
nonzero element. Then $V$ is atomic if and only if the vector
lattice $E$ is atomic.
\end{prop}
\begin{proof}
Let $V$ be  atomic. We shall  prove that $E$ is atomic too. First we
prove that a norm of an arbitrary atom in $V$ is also an atom on
$E$. Assume that $v\in V$ is an atom, $e=\ls v\rs\in E^{+}$ and
there are  elements $f_{1},f_{2}$, such that $0<f_{1}\leq e$,
$0<f_{2}\leq e$ and $f_{1}\bot f_{2}$. Then the band
$F_{1}=\{f_{1}\}^{\bot\bot}$ is disjoint to the band
$F_{2}=\{f_{2}\}^{\bot\bot}$ and by our assumption there exist two
nonzero elements $v_{1},v_{2}\in V$, such that $\ls v_{1}\rs\in
F_{1}$ and $\ls v_{2}\rs\in F_{2}$. Then we may write
$$
\ls v_{i}\rs=\ls v_{i}\rs\wedge f_{i}+e_{i};\,e_{i}\geq 0;\,
i\in\{1,2\}.
$$
By decomposability of the vector norm there exist $w_{1},w_{2}\in
V$, so that $0\leq\ls w_{i}\rs=\ls v_{i}\rs\wedge f_{i}\leq e=\ls
v\rs$,  $i\in\{1,2\}$ and $w_{1}\bot w_{2}$. But this is a
contradiction and therefore $\ls v\rs$ is an atom in $E$. Let $e\in
E^{+}$ and $e\bot\ls v\rs$ for every atom $v\in V$. Then there
exists a nonzero element $h\in V$, such that $\ls
h\rs\in\{e\}^{\bot\bot}$. Hence $e=0$ and $E$ is an atomic vector
lattice. The converse assertion is obvious.
\end{proof}

The following theorem is the second  main result of the our paper.

\begin{thm} \label{th:sdfah}
Let $(V,E),(W,F)$ be  lattice-normed spaces,  where  $W$ is atomic,
$E,F$  are vector lattices with the  the principal projection
property and  $E$ is  atomless. Then every dominated, order
continuous  linear operator $T: V \to W$ is order narrow.
\end{thm}

For the proof we need an auxiliary result that gives  a
representation of an element of an  atomic lattice-normed space via
atoms. Let $(W,F)$ be an  atomic lattice-normed space with a
generating collection of  atoms $(u_i)_{i \in I}$ and $F$ be a
vector lattice  with the principal projection property. Let
$\Lambda$ denote the directed set of all finite subsets of $I$
ordered by inclusion, that is, $\alpha \leq \beta$ for $\alpha,
\beta \in \Lambda$ if and only if $\alpha \subseteq \beta$. For
every $\alpha \in \Lambda$ we set
\begin{equation} \label{eq:apsow}
\mathbf{P}_\alpha = \sum_{i \in \alpha} P_{u_i},
\end{equation}
where $P_{u_i}$ is the band projection of $W$ onto the band
generated by the element $u_i$. It is immediate that $P_{u_{i}}$ is
a band projection of $F$ onto the band generated by the element $\ls
u_{i}\rs$ and $\mathbf{P}_\alpha$ is the band projection of $W$ onto
the band generated by $\{u_i: \, i \in \alpha\}$.

\begin{prop} \label{th:rhscb}
Let $(W,F)$ be an   atomic lattice-normed space, $F$ be
a vector lattice with the principal projection property and
$(u_i)_{i \in I}\subset W$ be a generating collection of  atoms. If
$g \in W$ then $P_{u_i} g = a_i$ for every $i \in I$ and some $a_i
\in \mathbb R$.
\end{prop}

\begin{proof}[Proof of Proposition~\ref{th:rhscb}]
$(1)$ Let $g\in W$. Then there exists a finite collection of atoms
$u_{1},\dots,u_{n}$, such that $\ls
g\rs\leq\sum\limits_{i=1}^{n}\lambda_{i}\ls u_{i}\rs$,
$\lambda_{i}\in\Bbb{R}_{+}$, $i\in\{1,\dots,n\}$. By decomposability
of the vector norm there exist  $g_{1},\dots, g_{n}$ in $W$, such
that
$$
g=\sum\limits_{i=1}^{n}g_{i};\,\ls g_{i}\rs\leq\lambda_{i}\ls
u_{i}\rs;\,i\in\{1,\dots,n\}.
$$
Moreover $g_{i}=a_{i}u_{i}$ for some $a_{i}\in\Bbb{R}$ and every
$i\in\{1,\dots,n\}$,  $\ls g_{i}\rs\bot\ls g_{j}\rs$, $j\neq i$.
Therefore $P_{u_i} g = a_i u_i$.
\end{proof}

\begin{proof}[Proof of Theorem~\ref{th:sdfah}]
Let $T: V \to W$ be a dominated  operator. Fix any $v \in V$ and
$e=\ls v\rs\in E^{+}$. Since the set $\mathfrak{F}_e$ of all
fragments of $e$ is order bounded in $E$, its image $\ls
T\rs(\mathfrak{F}_e)$ is order bounded in $F$, say, $\ls
Tx\rs\leq\ls T\rs\ls x\rs \leq f$ for some $f \in F^+$ and all $x
\sqsubseteq v$. Let $(u_i)_{i \in I}$ be a generating collection of
atoms of $W$, $\Lambda$ the directed set of all finite subsets of
$I$ ordered by inclusion, and $(\mathbf{P}_\alpha)_{\alpha \in
\Lambda}$ the net of band projections of $F$ defined by
\eqref{eq:apsow}. By  Proposition~\ref{th:rhscb},
$\mathbf{P}_\alpha$ is a finite rank operator for every $\alpha \in
\Lambda$. Being a band projection, $\mathbf{P}_\alpha$ is order
continuous. Then for each $\alpha \in \Lambda$ the composition
operator $S_\alpha = \mathbf{P}_\alpha \circ T$ is a finite rank
dominated, order continuous  operator which is strictly narrow by
Theorem~\ref{pr:relatedd8}. So, for each $\alpha \in \Lambda$ we
choose a decomposition $v = v_\alpha' \sqcup v_\alpha''$ with
$S_\alpha (v_\alpha') = S_\alpha (v_\alpha'')$. Then
\begin{align*}
\ls T (v_\alpha') - T (v_\alpha'')\rs &= \ls(I - \mathbf{P}_\alpha) \circ T (v_\alpha') - (I - \mathbf{P}_\alpha) \circ T (v_\alpha'')\rs\\
&\leq \ls(I - \mathbf{P}_\alpha) \circ T (v_\alpha')\rs + \ls(I - \mathbf{P}_\alpha) \circ T (v_\alpha'')\rs\\
&\leq (I - \mathbf{P}_\alpha) \, \ls T (v_\alpha')\rs + (I - \mathbf{P}_\alpha) \, \ls T (v_\alpha'')\rs\\
&\leq 2 \, (I - \mathbf{P}_\alpha) (f) \overset{(o)}\longrightarrow
0.
\end{align*}
\end{proof}


 \section{Acknowledgement}
 Authors are very grateful to the referee for his (her) valuable
 remarks and suggestions and Helen Basaeva for applying her wonderful expertise of \TeX \,
 to the final preparation of the text.


\end{document}